\documentclass[12pt,a4paper]{amsart}

\usepackage{graphicx}

\usepackage{amsmath,amssymb}                      

\usepackage[margin=1.4in]{geometry}

\usepackage{amsthm}

\newtheorem{proposition}{Proposition}

\title[Modelling covariances with variograms]{How to model the covariance structure in a spatial framework: variogram or correlation function?}

\author[Pistone and Vicario]{
Giovanni Pistone
\and
Grazia Vicario
}

\address{G.Pistone: de Castro Statistics, Collegio Carlo Alberto, Via Real Collegio 30, 10024 Moncalieri, Italy (E-mail: {\tt giovanni.pistone@carloalberto.org}). G. Vicario: DISMA {\it Luigi Lagrange}, Politecnico di Torino, Corso Duca degli Abruzzi 24, 10124 Torino, Italy (E-mail: {\tt grazia.vicario@polito.it})
}

\usepackage{etex}
\usepackage{bm}
\usepackage{localmacros}

\begin{document}
\thispagestyle{empty}
\maketitle             

\begin{abstract} The basic Kriging's model assumes a Gaussian distribution with stationary mean and stationary variance. In such a setting, the joint distribution of the spatial process is characterized by the common variance and the correlation matrix or, equivalently, by the common variance and the variogram matrix. We discuss in in detail the option to actually use the variogram as a parameterization.
\end{abstract}

\section{Introduction}
The present note is a development of the Author's paper \cite{pistone|vicario:2016sis2014}. An application was implemented in Vicario \emph{et al.} paper \cite{vicario|pistone|craparotta:2016}. In the interest of clarity we allow here for a little overlap with the papers referred to above.

We discuss the notion of \emph{variogram} as it is used in Geostatistics and we offer some preliminary thought about the possibility of a non-parametric approach to Universal Kriging aiming to the use of a Bayes methodology. Variograms  are due to Matheron \cite{matheron:1962}, and there are many modern expositions i.e., Cressie \cite[Ch. 2]{cressie:1993}, Chiles and Delfiner \cite[Ch. 2]{chiles|delfiner:1999}, Gneiting \emph{et al.} \cite{gneiting|sasvari|schlauder:2001}, Gaetan and Guyon \cite[Ch.1]{gaetan|guyon:2010}.

In Sec. \ref{sec:UK} we give a brief overview of the so called Universal Kriging model and its parameterization with the Matheron's variogram function. In Sec. \ref{sec:V} we define the general variogram matrix and give a necessary and sufficient condition for a positive variance $\sigma^2$ and a matrix $\Gamma$ to be a variogram matrix of a covariance $\sigma^2R$, where $R$ is a correlation matrix. In Sec. \ref{sec:inv} we provide some useful computations concerning the inverse variogram matrix. Sec \ref{sec:projection} is devoted to an interpretation of the variogram matrix as related to a projection of the Gaussian field. In Sec. \ref{sec:elliptope} we discuss the shape of the set of parameters of the general Kriging model. A section of conclusions closes the paper.

\section{Universal Krige setup} \label{sec:UK}

 We consider a Gaussian $n$-vector $Y$, $n\ge 2$, whose mean has the form 
  $\bmu = \mu \one$ and whose covariance matrix
  $\Sigma = [\sigma_{ij}]_{ij=1}^n$ has constant diagonal
  $\sigma_{ii} = \sigma^2$, $i=1,\dots,n$. The assumption on the mean and the diagonal terms is the weakest \emph{stationarity assumption}, i.e. a 1-st order stationarity. We can write $Y \sim \Normal_n(\mu\one,\sigma^2R)$, where $\mu$ is a general mean value, $\sigma^2$ the common variance,  and $R = [\rho_{ij}]_{i,j=1}^n$ a generic correlation matrix.

  The \emph{variogram} of $Y$ is the $n \times n$ matrix
  $\Gamma = [\gamma_{ij}]_{i,j=1}^n$ whose element $\gamma_{ij}$ is half the variance of the difference $Y_i - Y_j$. As the mean value is constant, the variance of the difference is equal to the second moment of the difference. It is expressed in terms of the common variance $\sigma^2$ and the correlations $\rho_{ij}$, $i,j = 1,\dots,n$, as
  \begin{multline*}
    2\gamma_{ij} = \varof{Y_i-Y_j} = \sigma^2(\sbasis_i - \sbasis_j)' R (\sbasis_i - \sbasis_j) = \\ \sigma^2\left(\rho_{ii} + \rho_{jj} - 2 \rho_{ij}\right) = 2 \sigma^2 (1-\rho_{ij}) \ ,
  \end{multline*}
and, in matrix form, as
  \begin{equation*}
    \Gamma = \sigma^2(\one \one' - R).
  \end{equation*}

The simple Gaussian model described above is commonly used in Geostatistics, when each random component $Y_i$ of the random vector $Y$ is associated to a location $x_i$, $i = 1,\dots,n$, in a given region $X$, $x_i \in X$, $i = 1,\dots,n$.

 We briefly describe the most common set-up in Geostatistics. The elements of the variogram matrix $\gamma_{ij}$ are assumed to be a given function $\gamma$ of the distance between two locations, $\gamma_{i,j} = \gamma(d(x_i,x_j))$. In such a case the statistical model is characterized by the choice of a distance $d(x,y)$, $x,y \in X$, and by the choice of a function $\gamma$, called \emph{variogram function}, defined on a real domain containing $\setof{d(x,y)}{x,y \in X}$ . The existence of a positive $\sigma^2$ and a correlation matrix $R = [\rho_{ij}]$ such that $\sigma^2(1-\rho_{ij}) = \gamma(d(x_i,x_j))$ imposes a nontrivial condition on the function $\gamma$, see Sasv{\'a}ri \cite{sasvari:1994book} and Gneiting \emph{et al.} \cite{gneiting|sasvari|schlauder:2001} and \cite{}. Such a model, where it is assumed that the vector of means is constant, $\bmu = \mu \one$ is called \emph{universal Krige} model. We do not consider the more general case of a non-constant mean. 

Krige has further qualified this model by adding assumptions on the variance function and suggesting a statistical method to estimate the value at an untried point $x_0$ given a set of observation $Y_1,\dots,Y_n$ at points $x_1,\dots,x_n$. Precisely:

\begin{enumerate}
  \item Krige's modeling idea is to assume the variogram function $\gamma$ to be an \emph{increasing} function on $[0,\infty[$, so that the variogram's values are increasing with the distance. Moreover, the correlation between locations is assumed to be positive. The rational is to model a variability which increases with the distance and is bounded by a general variance:
  \begin{equation*}
   0 \le \frac12\varof{Y_i-Y_j} = \gamma_{ij} = \gamma(d(x_i,x_j)) = \sigma^2(1-\rho_{ij}) \le \sigma^2 \ . 
  \end{equation*}

The increasing function $\gamma \colon \reals_+ \to \reals_+$ is assumed to be continuous everywhere but at 0. As it is bounded at $+\infty$, the general shape is as in Fig. \ref{fig:variogram}.
\begin{figure}[t]
  \centerline{
    \includegraphics[width=0.7\textwidth]{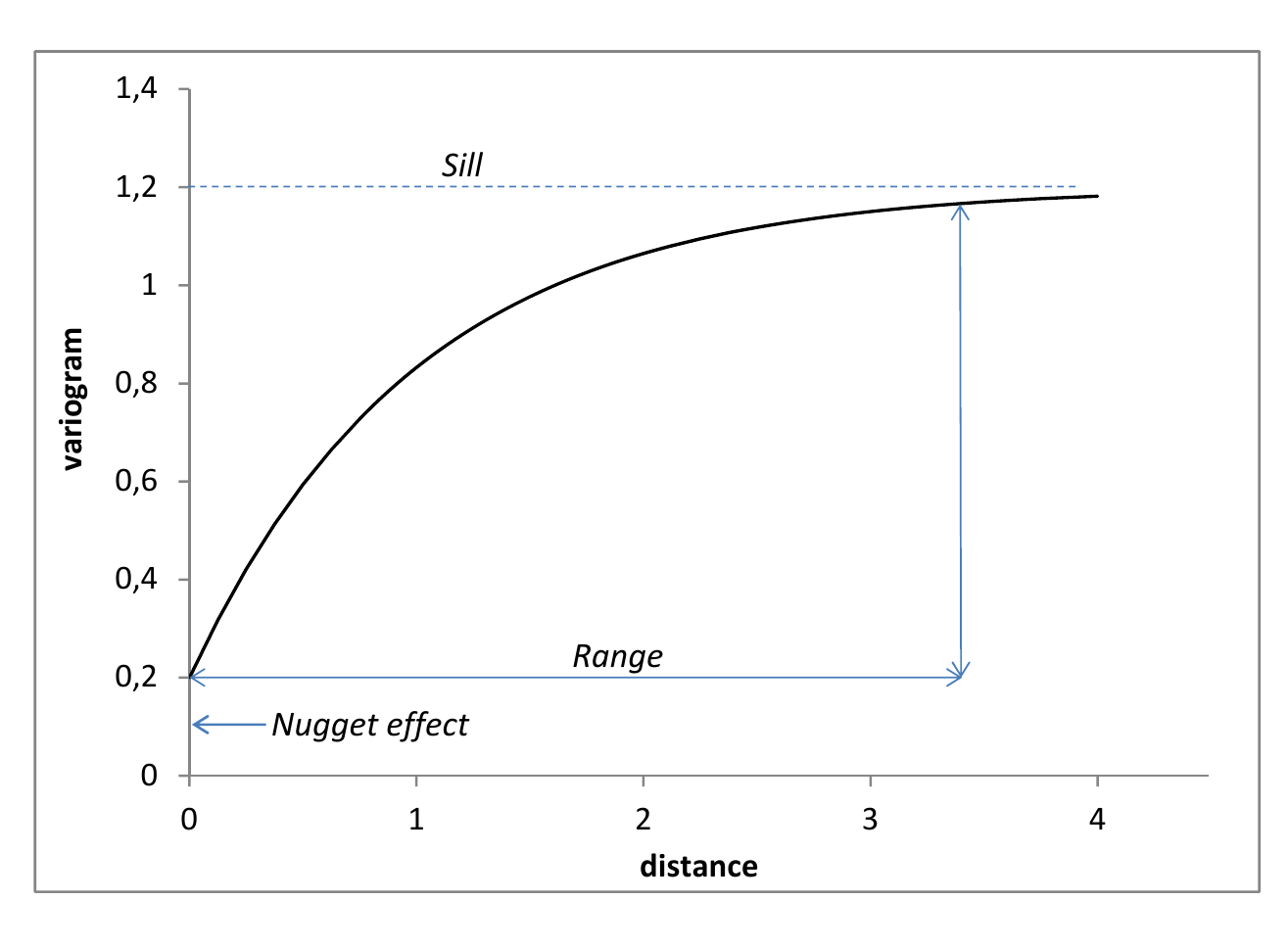}}
  \caption{A general variogram function is 0 at 0, can have jump at 0 which is called \emph{Nugget}, has a finite limit at $+\infty$ named \emph{Sill}. The \emph{Range} is a length such that the value is equal to the limit value for any practical purpose.}
  \label{fig:variogram}
\end{figure}

The parameters in the Krige's universal model are unrestricted values of $\mu \in \reals, \sigma^2 > 0$, and restricted values of $R$ that are usually estimated over a suitable parametric model. 

\item Krige's idea is to predict the value $Y_0=Y_{x_0}$ at an untried location $x_0$ with the conditional expectation based on a plug-in estimate of the parameters. If $I = \set{1,\dots,n}$ and the locations in the model are $x_0,x_1,\dots,x_n$, the regression value is 
  \begin{equation*}
  \widehat{Y_0 - \mu} = \Sigma_{0,I} \Sigma_{I,I}^{-1} (Y_I - \mu \one_I)
  , \quad \text{with} \quad  \Sigma = \begin{bmatrix} \Sigma_{I,I} & \Sigma_{I,0} \\ \Sigma_{0,I} & \sigma^2 \ .
    \end{bmatrix}
  \end{equation*}
  The set of data that give the same prediction is an affine
  plane in $\reals^n$. The variance of the prediction is $\sigma^2_0 - \Sigma_{0,I} \Sigma_{I,I}^{-1}\Sigma_{I,0}$. 
\end{enumerate}

In this paper we do not follow this approach, but we adopt a general non-parametric attitude, where $\mu$ is real number, $\sigma^2$ is a positive real number, $R$ is a positive definite matrix with unit diagonal, possibly with positive entries. The variogram matrix $\Gamma$ is not restricted and we do not enforce the existence of any special form.

\section{The variogram matrix} \label{sec:V}

Our plan now is to express the Krige's computations in term of Matheron's variogram matrix $\Gamma$. One good reason to use $\Gamma$ as a basic parameter is because its empirical estimator is unbiased.

Let us discuss in some detail the basic transformation of matrix parameters
\begin{equation}\label{eq:transform}
  \Gamma = \sigma^2(\one\one' - R) = \sigma^2\one\one' - \Sigma \ .
\end{equation}

We note that $\Gamma=0$ if, and only if, $R = \one\one'$, such extreme case being always excluded in the following. In fact, in most cases we will assume $\det R \ne 0$.

The entries of $\Gamma$ are nonnegative and bounded by $2\sigma^2$ because the correlations are bounded between $-1$ and $1$. If all the correlations are nonnegative, the the entries of $\Gamma$ are bounded by $\sigma^2$.

The difference between the covariance matrix and the variogram matrix is a matrix of rank 1, $\Gamma + \Sigma =  \sigma^2 \one\one'$. Let us remark moreover that 
\begin{equation} \label{eq:negdef}
\frac1n \one\one' = \frac1{n\sigma^2} (\Gamma + \Sigma)
\end{equation}
is the orthogonal projector on the space of constant vectors $\spanof{\one}$.

We denote by $\mathbb S_=$ the cone of nonnegative definite matrices with constant diagonal, by $\mathbb S_{=1}$ the convex set of correlation matrices, and by $\mathbb V$ the cone of variogram matrices. We have the following characterization of $\mathbb V$.

\begin{proposition} \label{prop:negdef}
A nonzero matrix $\Gamma$ is the variogram matrix of some covariance matrix of the form $\Sigma = \sigma^2 R$, with $\sigma^2 > 0$ and $R$ a correlation matrix, if, and only if, the three following conditions hold
  \begin{enumerate}
  \item \label{item:negdef1} $\Gamma$ is symmetric, and has zero diagonal; 
  \item \label{item:negdef2} $\Gamma$ is \emph{conditionally negative definite}, i.e. $\bw \Gamma \bw \le 0$ if $\bw'\one=0$; 
  \item \label{item:negdef3} $\sup \setof{\bx' \Gamma \bx}{\bx'\one = 1} \le \sigma^2$.
   \end{enumerate}
\end{proposition}

\begin{proof}
Assume $\Gamma=\sigma^2(\one\one' - R) \ne 0$ with $R$ correlation matrix and $\sigma^2 > 0$. Condition \ref{item:negdef1} follows from the definition. If we write a generic vector as $\bx = \bw + \alpha \one$ with $\bw'\one = 0$, we have
\begin{equation*}
 n^2 \sigma^2 \alpha^2 = \bx' \Gamma \bx + \bx' \Sigma \bx \ .
\end{equation*}
In particular, Condition \ref{item:negdef2} follows because $\alpha=0$ implies $\bx' \Gamma \bx = -  \bx' \Sigma \bx$. Finally, if $\bx'\one = 1$, that is $n \alpha=1$, we have $\bx'\Gamma\bx - \sigma^2 = \bx' \Sigma \bx \ge 0$ and Condition \ref{item:negdef3} follows.

Viceversa, consider the matrix $\one\one' - \sigma^{-2} \Gamma$. It is symmetric, with unit diagonal. We need only to show it is positive definite:
  \begin{multline*}
    \bx' (\one\one' - \sigma^{-2} \Gamma) \bx = (\bx'\one)^2 - \sigma^{-2} \bx' \Gamma \bx  = \\ (\bx'\one)^2(1 - \sigma^{-2} \left(\frac1{\bx'\one}\bx\right)' \Gamma \left(\frac1{\bx'\one}\bx\right) \ge 0
  \end{multline*}
because of Condition \ref{item:negdef3}.
\end{proof}

The lower bound imposed on $\sigma^2$ means that the parameterization with $\sigma^2$, carrying one degree of freedom, and with $\Gamma$, carrying $n(n-1)/2$ degrees of freedom, has a drawback in that the two parameters are not independently defined on a product set. Note the relation $\one'\Gamma\one = \sigma^2(n^2- \one' R \one)$ that we are going to discuss in Sec. \ref{sec:inv} below. 

In conclusion, there is a one-to-one transformation of parameters
\begin{equation*}
(\sigma^2,\Gamma) \leftrightarrow (\sigma^2, R) \leftrightarrow \Sigma
\end{equation*}
with $\sigma^2 \in \reals_>$, $\Gamma \in \mathbb V$, $R \in \mathbb S_{=1}$, $\Sigma \in \mathbb S_=$, namely:
\begin{enumerate}
\item
The mapping from $\Sigma \in \mathbb S_=$ to the couple $(\sigma^2,\Gamma) \in \reals_> \times \mathbb V$ factors as
\begin{equation*}
  \mathbb S_{=} \ni \Sigma \mapsto (\frac1n \traceof \Sigma,\left(\frac1n \traceof \Sigma\right)^{-2} \Sigma) = (\sigma^2,R) \in ]0,\infty[ \times \mathcal R 
\end{equation*}
and
\begin{multline*}
]0,\infty[ \times \mathcal R \ni (\sigma^2,R) \mapsto (\sigma^2,\sigma^2(\one\one' - R)) = \\  (\sigma^2,\Gamma) \in \setof{(\sigma^2,\Gamma)}{\Gamma\in\mathbb V,\sup \setof{\bx' \Gamma \bx}{\bx'\one = 1} \le \sigma^2}.   
\end{multline*}
\item Inverse is
\begin{multline*}
\setof{(\sigma^2,\Gamma)}{\Gamma\in\mathbb V,\sup \setof{\bx' \Gamma \bx}{\bx'\one = 1} \le \sigma^2} \ni (\sigma^2,\Gamma) \mapsto \\ \sigma^2 \one\one' - \Gamma = \Sigma \in \mathbb S_{=}
\end{multline*}
\end{enumerate}

\section{Inverse variogram matrix $\Gamma^{-1}$}
\label{sec:inv} 

The equations two above, both based on the definition of the variogram matrix as $\Gamma = \sigma^2(\one\one' - R)$, provide a simple connection between the parameterization based on the covariance matrix $\Sigma$ and the parameterization based on the couple $\sigma^2$ and $\Gamma$. However, we want to spell out the computation of an other key statistical parameter, namely the concentration matrix $\Sigma^{-1}$. We begin by recalling a well known equation in matrix algebra \cite{numericalrecipies:1996}. We review the result in detail as we need an exact statement of the conditions under which is true in our case.

\begin{proposition}[Sherman-Morrison formula] Assume the matrix $A$ is invertible. The matrix $\one\one' - A$ is invertible if, and only if, $\one'A^{-1}\one \ne 1$. In such a case, 
  \begin{equation*}
    \det (\one\one' - A) = (-1)^n  (1 - \one' A^{-1} \one) \det A \ ,
  \end{equation*}
\begin{equation*}
      (\one\one' - A)^{-1} = -A^{-1}-(1-\one'A^{-1}\one)^{-1}A^{-1}\one\one'A^{-1}.
    \end{equation*}
  \end{proposition}
  
\begin{proof} The multi-linear expansion of $\det (\one\one' - A)$ is written in terms of the adjoints $(-A)^{ij}$ of each element $(-A)_{ij}$ by  

\begin{align*}
  \det (\one\one' - A) &= \detof{-A} + \sum_{j=1}^n \sum_{i=1}^n  (-A)^{ij}
\\ &= (-1)^n \Det A - (-1)^{n-1} \sum_{i,j=1}^n A^{ij} \\ &= (-1)^n \Det A - (-1)^{n-1} \one (\Adj A) \one' \ .
\end{align*}

As $\det A \ne 0$, we can factor-out $(-1)^n (\det A)$ to get

\begin{equation*}
  \det (\one\one' - A) = (-1)^n (\det A) (1 - \one' A^{-1} \one) \ 
\end{equation*}
and the statement about the determinant follows. The inversion formula is directly checked.
\end{proof}

We are concerned with the invertibility of $\Gamma = \sigma^2(\one\one' - R)$, hence we need to discuss the condition $\one' R^{-1} \one \ne 1$.

\begin{proposition}
\label{prop:3}
Let $R$ be a correlation matrix and assume $\det R \ne 0$. Let $\lambda_j > 0$, $j=1,\dots,n$, be the spectrum of $R$ and let $\bu_j$ be a set of unit eigenvectors. It can be proved: 
\begin{enumerate}
\item $\Tr R = \sum_{j=1}^n \lambda_j = n$ and $\det R = \prod_{j=1}^n \lambda_j \le 1$, with equality if, and only if, $R = I$.
\item $\Tr R^{-1} = \sum_{j=1}^n \lambda_j^{-1} \ge n$ with equality if, and only if, $R=I$.
\item $\one' R^{-1} \one \ne 1$.
\end{enumerate}
\end{proposition}

\begin{proof}
\begin{enumerate}
\item  $n = \Tr R = \sum_{j=1}^n \lambda_j$. From $\Det R = \prod_{j=1}^n \lambda_j$, as the arithmetic mean is larger than the the geometric mean, 
\begin{equation*}
  1 = \frac{\sum_{j=1}^n \lambda_j}n \ge \left(\prod_{j=1}^n \lambda_j\right)^{\frac1n} = (\Det R)^{\frac1n} \ ,
\end{equation*}
with equality if, and only if the $\lambda_j$'s are all equal, hence equal to 1, which happens if $R = I$. 
\item The geometric mean is larger or equal than the harmonic mean, hence
\begin{equation*}
1 \ge  (\Det R)^{\frac1n} = \left(\prod_{j=1}^n \lambda_j\right)^{\frac1n} \ge n \left(\sum_{j=1}^n \lambda^{-1}_j\right)^{-1} \ , 
\end{equation*}
with equality if, and only if, $\lambda_j = 1$, $j=1,\dots,n$. It follows $\frac1n \sum_{j=1}^n \lambda_j^{-1} \ge 1$.
\item  We derive a contradiction from $1 = \one' R^{-1} \one$. As $R^{-1} = \sum_{j=1}^n \lambda_j^{-1} \bu_j \bu_j'$ and $\sum_{j=1}^n (\one' \bu_j)^2 = \normof \one ^2 = n^2$,
  \begin{equation*}
    1 = \one' R^{-1} \one = \sum_{j=1}^n \lambda_j^{-1} (\one'\bu_j)^2 = n^2 \sum_{j=1}^n (\lambda_j)^{-1} \theta_j \ ,
  \end{equation*}
where $\theta_j = (\one'\bu)^2 / n^2 \ge 0$ and $\sum_{j=1}^n \theta_j = 1$. From the convexity of $\lambda \mapsto \lambda^{-1}$ we obtain
\begin{equation*}
1 = n^2 \sum_{j=1}^n (\lambda_j)^{-1} \theta_j \ge n^2 \left(\sum_{j=1}^n \lambda_j \theta_j\right)^{-1} \ , 
\end{equation*}
hence the contradiction
\begin{equation*}
  1 \le \frac1{n^2} \sum_{j=1}^n \lambda_j \theta_j \le \frac1{n^2} \max\setof{\lambda_j}{j=1,\dots,n}  \le \frac 1n \ .
\end{equation*}
\end{enumerate}
\end{proof}

\bigskip

From Proposition \ref{prop:3} we have immediately the following result of interest.

\begin{proposition} Assume the correlation matrix $\Sigma=\sigma^2 R \in \mathbb S_{=}$ is invertible. It follows that $\Gamma = \sigma^2(\one\one' - R)$ is invertible, with
  \begin{equation}\label{eq:SM1}
 \Gamma^{-1} = - \Sigma^{-1} - (\sigma^{-2} - \one' \Sigma^{-1} \one)^{-1} \Sigma^{-1}\one \one' \Sigma^{-1}
\end{equation}
and
\begin{equation}\label{eq:SM2}
\Sigma^{-1} = -\Gamma^{-1} - (\sigma^{-2} - \one' \Gamma^{-1} \one)^{-1} \Gamma^{-1}\one \one' \Gamma^{-1}
\end{equation}
\end{proposition}

\begin{proof}
From the assumption it follows $\det R \ne 0$ so that
  \begin{equation*}
    \sigma^{-2} - \one' \Sigma^{-1} \one = \sigma^{-2}(1 - \one' R \one) \ne 0 \ ,
  \end{equation*}
hence the conclusion.
\end{proof}

We can now analyze the likelihood of the Gaussian model N$(\mu\one,\sigma^2R)$ in terms of the variogram.

First, we compute the determinant of the correlation matrix
\begin{align*}
    \detof{\sigma^2 R} &= \detof{\sigma^2 \one\one' - \Gamma} \\
&= \sigma^{2n} \detof{\one\one' - \sigma^{-2} \Gamma} \\ &= \sigma^{2n} \left[\detof{-\sigma^{-2}\Gamma} + \one' \adjof{-\sigma^{-2} \Gamma} \one \right] \\ &= \detof{-\Gamma} - \sigma^2 \one' \adjof{-\Gamma} \one \ .
\end{align*}

Second, we compute the quadratic form of the concentration matrix

  \begin{align*}
    \by' \Sigma^{-1} \by &= \by'\left(-\Gamma^{-1} - (\sigma^{-2} - \one' \Gamma^{-1} \one)^{-1} \Gamma^{-1}\one \one' \Gamma^{-1}\right) \by \\ &= - \by' \Gamma^{-1} \by - (\sigma^{-2} - \one' \Gamma^{-1} \one)^{-1} (\by' \Gamma^{-1} \one)^2. 
  \end{align*}

Third, we compute the log-likelihood with $\mu=0$:

  \begin{multline*}
    \log p\left(\by \middle| \sigma^2, \Gamma \right) = \\ 
  -\frac{n}2\logof{2\pi} 
  -\frac12 \logof{\detof{\one\one' - \sigma^{-2} \Gamma}} - \frac1{2\sigma^2} \by' (\one\one' - \sigma^{-2} \Gamma)^{-1} \by = \\ -\frac{n}2\logof{2\pi}
  -\frac12 \logof{\detof{-\Gamma} - \sigma^2 \one' \adjof{-\Gamma} \one} \\ + \frac12 \by' \Gamma^{-1} \by +\frac12 (\sigma^2 - \one' \Gamma^{-1} \one)^{-1} (\by' \Gamma^{-1} \one)^2 \ .
  \end{multline*}

Here are the essentials of the computations leading to a maximum likelihood estimation of $\Gamma$ are the following. In the direction of a generic symmetric matrix with zero diagonal $H$,

  \begin{equation*}
    d_H(\Gamma \mapsto \logof{\detof{\one\one' - \sigma^{-2}\Gamma}}) = \traceof{(\sigma^2 \one\one' - \Gamma)^{-1} H}
  \end{equation*}
and
  \begin{multline*}
    d_H(\Gamma \mapsto \by' (\one\one' - \sigma^{-2}\Gamma)^{-1}\by) \\ = \sigma^2 \traceof{(\sigma^2 \one\one' - \Gamma)^{-1} \by \by' (\sigma^2 \one\one' - \Gamma)^{-1} H}\ .
  \end{multline*}
so that the normal equations for $\Gamma$ reduce to the condition
  \begin{equation*}
   -(\sigma^2 \one\one' - \Gamma)^{-1} + (\sigma^2 \one\one' - \Gamma)^{-1} \by \by' (\sigma^2 \one\one' - \Gamma)^{-1} \quad \text{is diagonal} \ .  
 \end{equation*}

The approach with parameters $\sigma^2$, $\Gamma$ is feasible in principle, but it does not appear promising in term of ease of computation. We will see in Sec. \ref{sec:projection} below a different, possibly better, approach.

\section{Projecting on $\spanof \one ^\perp$}\label{sec:projection}

We now change our point of view to consider the same problem from a different angle suggested by the observation the variogram does not change if we change the general mean $\mu$. In fact, we can associate the variogram with the
state space description of the Gaussian vector.

The following proposition is a new characterization of the variogram matrix in our setting.

\begin{proposition}
\label{prop:split}
\begin{enumerate}
\item
The matrix $\Gamma$ is a variogram matrix of a covariance matrix $\Sigma \in \mathbb V_+$ if, and only if, the matrix 
  \begin{equation}\label{eq:NCcondition}
    \Sigma_0 = - \left(I-\frac1n \one\one'\right)' \Gamma \left(I-\frac1n \one\one'\right)
  \end{equation}
  is symmetric, positive definite and with constant diagonal.
\item If $Y_0 \sim \Normal_n(0,\Sigma_0)$, then its variogram is $\Gamma$ and $Y_0$ is supported by $\spanof{\one}^\perp$. 
\end{enumerate}
\end{proposition}

\begin{proof}
  \begin{enumerate}
\item
If $\Gamma = \sigma^2(\one\one'-R)$ is the variogram matrix of $\Sigma=\sigma^2R$, then from Eq. \eqref{eq:NCcondition} we have
\begin{equation*}
  \Sigma_0 = \sigma^2 \left(I-\frac1n \one\one'\right)' R \left(I-\frac1n \one\one'\right),
\end{equation*}
which is indeed positive definite. Let us show that the diagonal elements of $\Sigma_0$ are constant.
\begin{align*}
  (\Sigma_0)_{ii} &= \sigma^2 \sbasis_i'\left(I-\frac1n \one\one'\right)' R \left(I-\frac1n \one\one'\right) \sbasis_i \\
&= \sigma^2 \left(\sbasis_i-\frac1n\one\right)' R \left(\sbasis_i-\frac1n\one\right) \\
&= \sigma^2 \left(\sbasis_i' R \sbasis_i - \frac 2n\sbasis_i R \one + \frac1{n^2} \one' R \one \right) \\
&= \sigma^2\left(\frac1{n^2} \one' R \one - 1\right)
\end{align*}

Viceversa, assume $\Sigma_0$ is a covariance matrix. As $\sbasis_i - \sbasis_j \in \spanof{\one}^\perp$, the variogram of $\Sigma_0$ has elements
\begin{multline*}
 (\sbasis_i - \sbasis_j)' \Sigma_0 (\sbasis_i - \sbasis_j) = \\ (\sbasis_i - \sbasis_j)' \left(I-\frac1n \one\one'\right)' (-\Gamma) \left(I-\frac1n \one\one'\right) (\sbasis_i - \sbasis_j) = \\ - (\sbasis_i - \sbasis_j)' \Gamma (\sbasis_i - \sbasis_j) = - \gamma_{ii} - \gamma_{jj} + 2 \gamma_{ij} = 2 \gamma_{ij}.
\end{multline*}
\item
  As $\one' (\sbasis_i - \sbasis_j)  = 0$, then $\one' \left(I-\frac1n \one\one'\right)' (-\Gamma) \left(I-\frac1n \one\one'\right) \one = 0$, hence the distribution of $Y_0$ is supported by the space $\spanof{\one}^\perp$.
\end{enumerate}
\end{proof}

It is possible to split every Gaussian $Y$ with covariance matrix $\Sigma \in \mathbb S_=$ according the splitting $\reals^n = \spanof\one \oplus \spanof\one^\perp$. The corresponding projections split the Gaussian process $Y$ into two components, one with the covariance as in Proposition \ref{prop:split}, the other proportional to the empirical mean. Note that the two components have a singular covariance matrix.

\begin{proposition}
Let $Y \sim \Normal_n(\bmu,\Sigma)$, $\Sigma = \sigma^2 R \in \mathbb S_+$ with variogram $\Gamma = \sigma^2(\one\one' - R)$. Let $\widehat Y = \left(I-\frac1n \one\one'\right)Y \sim \Normal_n\left(\bold 0, \Sigma_0\right)$ be the projection of $Y$ onto $\spanof{\one}^\perp$ so that we can write $Y = \widehat Y + \overline Y$, where each component of $\overline Y$ is the empirical mean $\frac1n \one' Y$. 
\begin{enumerate}
\item The distribution of $\widehat Y$ depends on the variogram only,
  \begin{equation*}
    \widehat Y \sim \Normal_n\left(\bold 0,- \left(I-\frac1n \one\one'\right)' \Gamma \left(I-\frac1n \one\one'\right)\right) \ ,
  \end{equation*}
  and the variogram matrix of $\widehat Y$ is $\Gamma$.
\item The distribution of $\frac1n \one' Y$, conditionally to $\widehat Y$ is Gaussian with mean $\mu$.
\end{enumerate}
\end{proposition}

\begin{proof}
    \begin{equation*}
      - \left(I-\frac1n \one\one'\right)' \Gamma \left(I-\frac1n \one\one'\right) = \left(I-\frac1n \one\one'\right)' \Sigma \left(I-\frac1n \one\one'\right) \ .
    \end{equation*}
  \end{proof}
  
This suggests the following empirical estimation algorithm.

\begin{enumerate}
\item Project the independent sample data $\by_1,\dots,\by_N$ onto $\spanof\one^{\perp}$ by subtracting the empirical mean $\by_+ = \frac1n \sum \by_i$, to get $\widehat \by_1 = \by_1 - \by_+,\dots,\widehat \by_N = \by_N - \by_+$. Use the empirical estimator of the variogram matrix on the projected data.
\item Estimate $\mu$ with the empirical mean.
\end{enumerate}

In the same spirit, we could suggest the simulation of a random variable with variogram matrix $\Gamma$ by the generation of $\Normal_n(\bold 0, \Sigma_0)$ data in $\spanof\one^\perp$.

Both suggestions will be further discussed in future work.

\section{Elliptope}\label{sec:elliptope}

We now turn to the geometrical description of the set of variograms. From the basic equation $\Gamma=\sigma^2(\one\one' - R)$ it follows that the set of variogram matrices is an affine image of the set of correlation matrices in the space of symmetric matrices. The set of correlation matrices is a convex bounded set whose geometrical shape has been studied in a number of papers, i.e. \cite{rousseeuw|molenberghs:1994}, \cite{rousseeuw|molenberghs:1994}, \cite{rapisarda|brigo|mercurio:2007}.   Such a shape is of central interest in a non parametric Bayesian approach to the statistics of the universal Kriging model. It appears also in convex optimization, where it has been called \emph{elliptope}.

Let us discuss the case $n=3$.
\begin{figure}[t]
\centering
\begin{tabular}{cc}
\begin{minipage}{160pt}
\includegraphics[height=155pt]{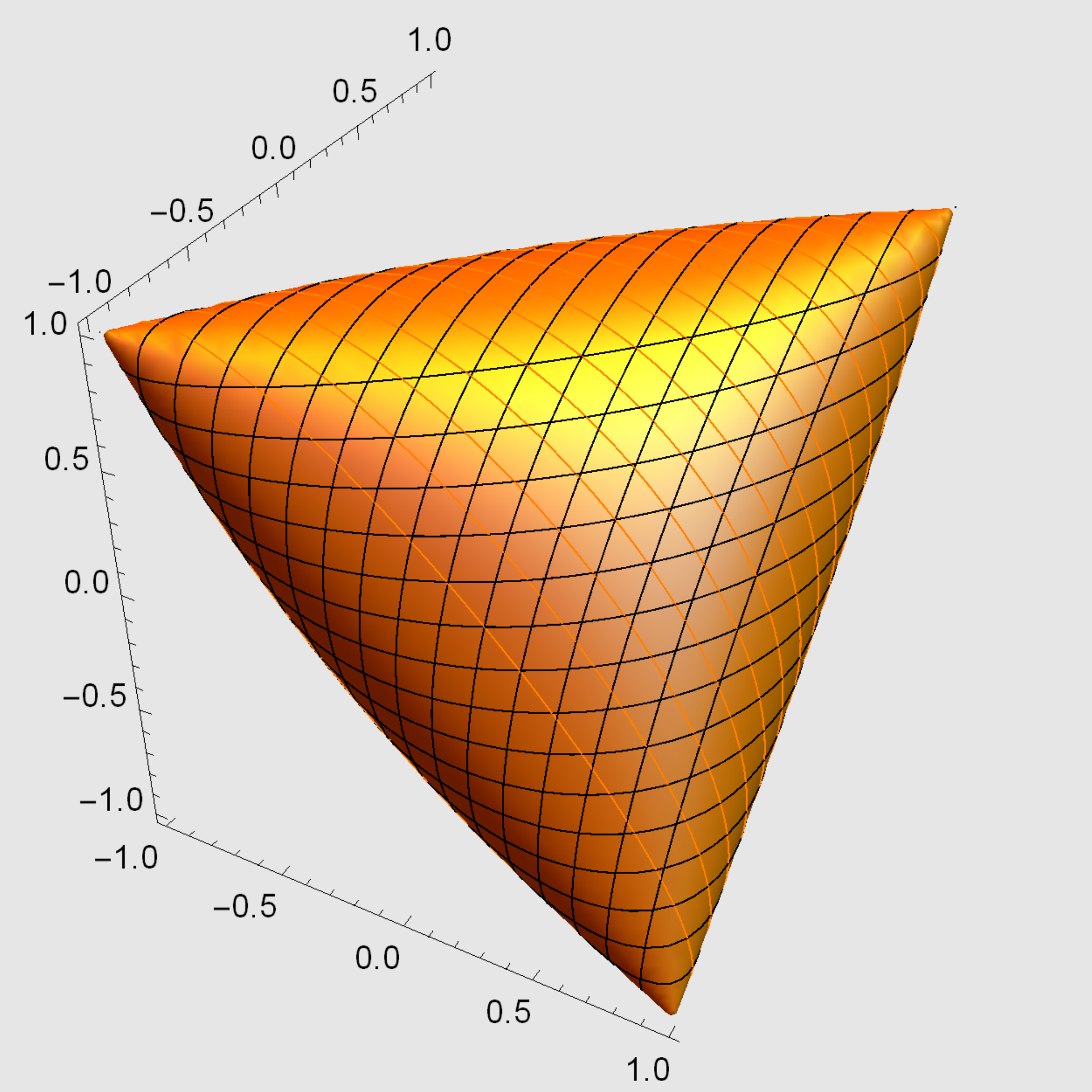}
\caption{The 3-elliptope}
\label{fig:3-elliptope}
\end{minipage}
&
\centering
\begin{minipage}{160pt}
\includegraphics[height=155pt]{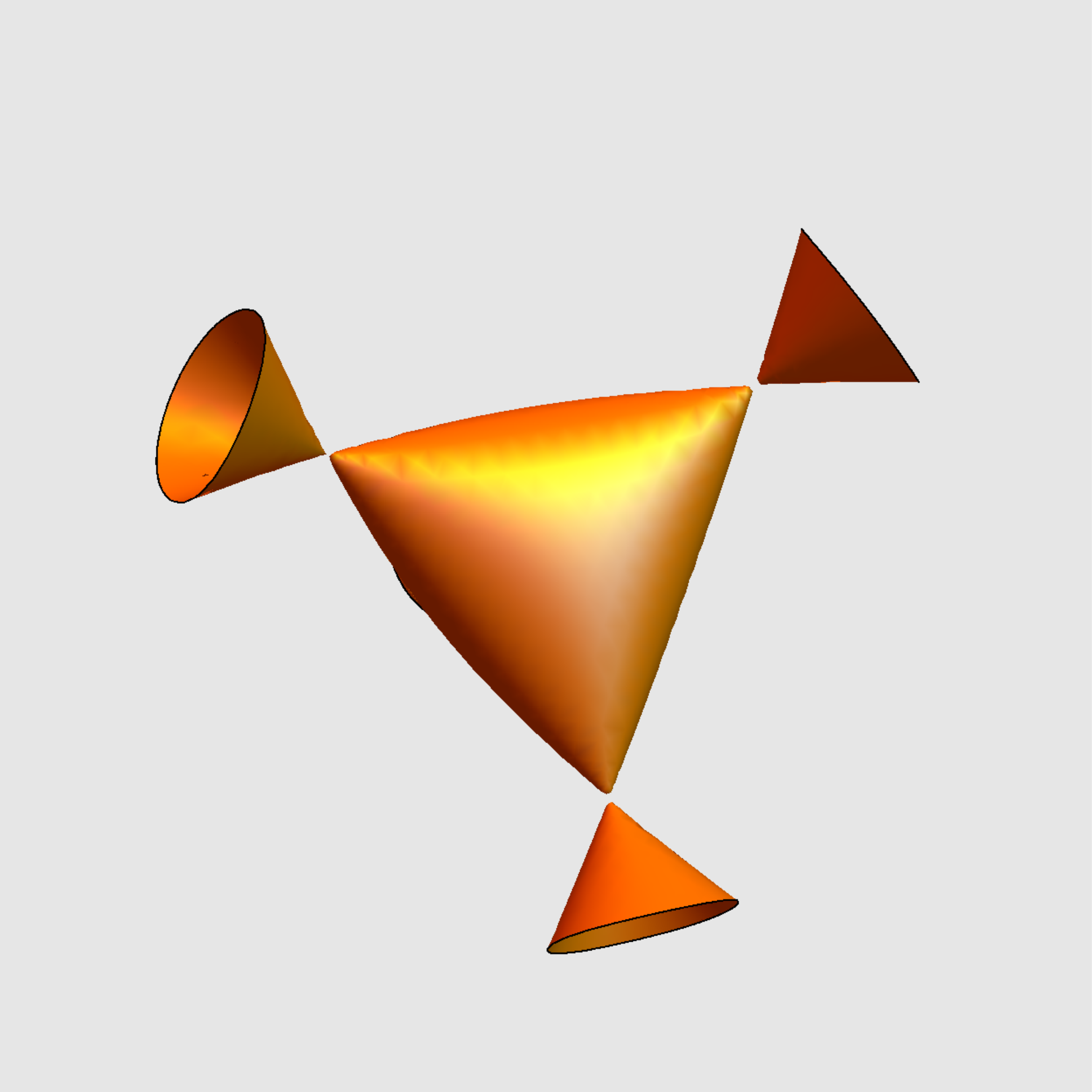}
\caption{Algebraic variety}
\label{fig:3-variety}
\end{minipage}
\end{tabular}
\end{figure}
All principal minors of $R$ are nonnegative,
  \begin{equation*}
    \detof R = \detof{\begin{bmatrix} 1 & x & y \\ x & 1 & z \\ y & z & 1
  \end{bmatrix}} = 1 - x^2 - y^2 - z^2 + 2 xyz \ge 0
  \end{equation*}
and $1-x^2,1-y^2,1-z^2 \ge 0$. The three last inequalities define the cube $Q = [-1,+1]^3$ while the equation $1 - x^2 - y^2 - z^2 + 2 xyz = 0$ is a cubic algebraic variety whose intersection with the cube $Q$ is the border of the elliptope. All horizontal section $z = c$, $-1 \le c \le 1$, of the elliptope are the interior of the ellipses
  \begin{equation*}
    1 - x^2 - y^2 + 2cxy \ge c^2
  \end{equation*}
Same for other sections. See Fig.s \ref{fig:3-elliptope} and \ref{fig:3-variety}.

Various proposals of apriory distribution on the elliptope exist, see for example \cite{barnard|mcculloch|meng:2000}. 

Going on with the discussion of our example, the volume is easily computed and the uniform apriori is defined. Simulation is feasible for example by rejection method. An other option is to write $R = A'A$ where the columns of A are unit vectors. This gives an other possible apriori starting from independent unit vectors. Simulation is feasible, for example starting with independent standard Gaussians.

An interesting option is the Cholesky representation. A symmetric matrix $A$ is positive definite if there exists an upper triangular matrix
  \begin{equation*}
    T =
    \begin{bmatrix}
      \bt_1' \\ \bt_2' \\ \bt_3'
    \end{bmatrix}
  =
  \begin{bmatrix}
    t_{11} & t_{12} & t_{13} \\ 0 & t_{22} & t_{23} \\ 0 & 0 & t_{33}
  \end{bmatrix}, \quad t_{ii} \ge 0
  \end{equation*}
such that
  \begin{equation*}
    A = T'T = \left[ \bt_i \cdot \bt_j\right]_{ij} =
    \begin{bmatrix}
      t_{11}^2 & t_{11}t_{12} & t_{11}t_{13} \\ t_{11}t_{12} & t_{12}^2 + t_{22}^2 & t_{12}t_{13}+t_{22}t_{23} \\ t_{11}t_{13} & t_{12}t_{13}+t_{22}t_{23} & t_{13}^2 + t_{23}^2 + t_{33}^2
    \end{bmatrix}
  \end{equation*}

Moreover, $t_{11}t_{22}t_{33} \ne 0$ $\Leftrightarrow$ $T$ is unique and invertible if, and only if, $A$ is invertible. It is an identifiable parameterization for non singular matrices.

In the case of a the correlation matrix $R = T'T$ with
  \begin{equation*}
    T =
    \begin{bmatrix}
  \bt_1' \\ \bt_2' \\ \bt_3'
    \end{bmatrix}
  =
  \begin{bmatrix}
    \sqrt{1-t_{12}^2-t_{13}^2} & t_{12} & t_{13} \\ 0 & \sqrt{1-t_{23}^2} & t_{23} \\ 0 & 0 & 1
  \end{bmatrix}, \quad \bt_i \in 0^{i-1} \times S_{n-1+1}^+
  \end{equation*}

It follows
  \begin{equation*}
    R =
    \begin{bmatrix}
  1 & \sqrt{1-t_{12}^2-t_{13}^2}t_{12} & \sqrt{1-t_{12}^2-t_{13}^2} t_{13} \\ \sqrt{1-t_{12}^2-t_{13}^2}t_{12} & 1 & t_{12}t_{13}+\sqrt{1-t_{23}^2}t_{23} \\ \sqrt{1-t_{12}^2-t_{13}^2}t_{13} & t_{12}t_{13}+\sqrt{1-t_{23}^2}t_{23} & 1
    \end{bmatrix}
  \end{equation*}
and
  \begin{equation*}
    \detof{R} = (1-t_{12}^2-t_{13}^2)(1-t_{23}^2)
  \end{equation*}

  \section{Conclusion}
  In the last decades Kriging models have been recommended not only for the original application, but spatial noisy data in general. Thanks to the availability of comprehensive computing facilities and the recent progresses in software development, the numerical simulation of technologically complex systems has become an attractive alternative option to the physical experimentation. The most popular meta-model used when dealing with Computer Experiments (CE) is the Kriging model. The accuracy of this model strongly depends on the detection of the correlation structure of the responses. In the Bayesian approach, where the posterior distribution of a prediction Krige's $Y_0$ given the training set $(Y_1,\dots,Y_n)$ requires less uncertainty as possible on the correlation function, the use of the variogram as a parameter should be preferred because it does not demand a parametric approach as the correlation estimation does. The authors proved in a previous paper \cite{pistone|vicario:2016sis2014} the equivalence between variogram and spatial correlation function for stationary and intrinsically stationary processes. Here the study has been devoted to the characterization of matrices which are admissible variograms in the case of 1-stationarity. We expect these findings will allow for the imputation of an apriori distribution on the set of variogram matrices, as Bayesians do for the correlation in the Kriging modernization.

\section*{Acknowledgments}
An early version of this paper was presented at the 4th Stochastic Modeling Techniques and Data Analysis International Conference, June 1--4, 2016,
University of Malta, Valletta, Malta, with the title \emph{Bayes and Krige: Generalities}. The Authors thank both Guillaume Kon Kam King (CCA, Moncalieri) and Luigi Malag\`o (CCA, Moncalieri) for suggesting relevant references. We thank Emilio Musso (DISMA, Politecnico di Torino) for help with the elliptopes picture. G. Pistone is supported by de Castro Statistics, Collegio Carlo Alberto, Montalieri, and he is a member of GNAFA-INDAM.

\def\cprime{$'$}

\end{document}